\documentclass[12pt]{amsart}
\usepackage{pdfsync}
\usepackage[utf8]{inputenc}
\usepackage[T1]{fontenc}
\usepackage{amsthm,amsmath,amssymb}
\usepackage[backend=bibtex,style=numeric,doi=false,url=false]{biblatex}
\bibliography{eigene,andere,arxiv,mathscinet}
\usepackage{hyperref}
\usepackage{cleveref}
 \newtheorem{Thm}{Theorem}[section]
 \newtheorem{Lem}[Thm]{Lemma}
 \newtheorem{Prop}[Thm]{Proposition}
 \newtheorem{Cor}[Thm]{Corollary}
\theoremstyle{remark}
 \newtheorem{Expl}[Thm]{Example}

 \newtheorem{Rem}[Thm]{Remark}
\theoremstyle{definition}
 
\numberwithin{equation}{section}

\newcommand\Z{\mathbb Z}
\newcommand\CC{\mathbb C}
\newcommand\ol[1]{\overline{#1}}

\renewcommand\S{\mathbb S}
\newcommand\A{\mathbb A}

\newcommand\disj{\sqcup}

\newcommand\adhit{\triangleright}

\newcommand\inv{^{-1}}
\def\HM#1.#2.#3.#4.{{^{#1}_{#3}\mathcal M^{#2}_{#4}}}

\newcommand\Ind{\operatorname{Ind}}
\newcommand\Stab{\operatorname{Stab}}
\newcommand\Rep{\operatorname{Rep}}
\newcommand\Irr{\operatorname{Irr}}

\newcommand\Indtobim{\mathcal F}

\newcommand\C{\mathcal C}

\newcommand\Cat[2]{\C(#1,#2)}
\newcommand\Obj[2]{\Indtobim_{#1}(#2)}
\newcommand\supp{\operatorname{supp}}
\newcommand\ord{\operatorname{ord}}
\newcommand\subgroup{\leq}
\begin{document}
 \title[]{Frobenius-Schur indicators for some fusion categories associated to symmetric and alternating groups}
 \author{Peter Schauenburg}
 \address{Institut de Math{\'e}matiques de Bourgogne --- UMR 5584 du CNRS\\
 Universit{\'e} de Bourgogne\\
 Facult{\'e} des Sciences Mirande\\
 9 avenue Alain Savary\\
 BP 47870 21078 Dijon Cedex\\
 France
 }
 \email{peter.schauenburg@u-bourgogne.fr}
 \subjclass[2010]{18D10,16T05,20C15}
 \keywords{Fusion category,Hopf algebra,Frobenius-Schur indicator}
 \begin{abstract}
 We calculate Frobenius-Schur indicator values for some fusion categories obtained from inclusions of finite groups $H\subset G$, where more concretely $G$ is symmetric or alternating, and $H$ is a symmetric, alternating or cyclic group. Our work is strongly related to earlier results by Kashina-Mason-Montgomery, Jedwab-Montgomery, and Timmer for bismash product Hopf algebras obtained from exact factorizations of groups. We can generalize some of their results, settle some open questions and offer shorter proofs; this already pertains to the Hopf algebra case, while our results also cover fusion categories not associated to Hopf algebras.
 \end{abstract}
 \maketitle

\section{Introduction}
\label{sec:introduction}
Having recently found general formulas for computing higher Frobenius-Schur indicators in group-theoretical fusion categories \cite{Sch:CHFSIFCCIFG,2015arXiv150202906S}, we take a step back, as it were, and put the formulas from \cite{Sch:CHFSIFCCIFG} to work in the special case of degree two indicators. Thus, our examples pertain to the special case of a group-theoretical fusion category $\C(G,H)$ defined ``without cocycles'' from a finite group $G$ and a subgroup $H$; this category can be viewed as the category of $G$-graded vectors spaces endowed with a two-sided action of the subgroup $H$ compatible with the grading.

The general formulas from \cite{Sch:CHFSIFCCIFG,2015arXiv150202906S} for arbitrary degree indiators feature a summation of character values over a set of group elements that does not seem to have a particularly transparent structure. In the case of degree two indicators, however, things simplify considerably: If the sum is not empty outright, then it can be written as a sum over all the elements of a certain (stabilizer) subgroup. These sums are much easier to handle, in particular they allow the indicators for the fusion category $\C(G,H)$ to be written as twisted indicators in the sense of \cite{MR1078503}.

In the case where the subgroup $H\subset G$ is part of an exact factorization $G=HQ$ with a second subgroup $Q\subset G$, the category $\C(G,H)$ is equivalent by \cite{MR1887584} to the module category of a bismash product Hopf algebra $\CC^Q\#\CC H$. (In the presence of cohomological data, one can similarly obtain bicross product Hopf algebras.) For such bismash product Hopf algebras, many results on indicator values were previously obtained (and some questions asked) by other authors \cite{MR1919158,MR2471448,2014arXiv1412.4725}, and these results have provided part of the motivation for the present paper. As it turns out, the indicator formulas for $\C(G,H)$ are not only applicable in more general situations (where a direct factorization is not available), but they also offer a significant advantage in the special case of a bismash product. The reason seems to be that objects in $\C(G,H)$ are described in terms of the double cosets of $H$ in $G$. Obtaining statements on the indicators is sometimes made easier by a good choice of double coset representative. Having a direct factorization $G=HQ$ (needed to obtain a bismash product Hopf algebra) implicitly imposes a choice of coset representatives (namely, among the elements of $Q$) which may not be advantageous.

\section{Preliminaries}
\label{sec:preliminaries}

For an inclusion of finite groups $H\subset G$, recall that simple objects in the group-theoretical fusion category $\Cat GH:=\HM G..H.H.=\C(G,H,1,1)$ are parametrized by pairs $(g,\chi)$ with $\chi\in\Irr(S(g))$, where $S(g)=\Stab_{H}(gH)=H\cap g\adhit H$. The indicators of such a simple are given by \cite{Sch:CHFSIFCCIFG}:
\begin{equation}\label{eq:24}
  \nu_m(g,\chi,H)=\nu_m(g,\chi)=\frac 1{|S(g)|}\sum_{\substack{x\in H\\(gx)^m\in H}}\ol\chi((gx)^m).
\end{equation}
Here, we have adopted a somewhat lighter notation than in \cite{Sch:CHFSIFCCIFG}, where $\Indtobim_g\colon\Rep(S(g))\to \Cat GH$ was used to denote the functor translating group representations to objects in the relevant category. However, we do have need sometimes to recall the subgroup $H$ with respect to which the indicator is calculated (while the group $G$ does not influence the result except by providing a place for $g$ to live in).

Note that $(gx)^m\in H\Leftrightarrow (gx)^m\in S(g)$.

The simple objects associated to two elements $g,g'$ are the same if and only if $g$ and $g'$ lie in the same double coset in $H\backslash G/H$. More precisely the two stabilizer groups $S(g)$ and $S(g')$ are then conjugate in $H$, and $\nu_2(g,\chi)=\nu_2(g',\chi')$ where $\chi'$ is obtained from $\chi$ by conjugation.

\section{General Results}
\label{sec:general-results}

The following remark is an obvious consequence of the indicator formula above.
\begin{Rem}\label{rem:2}
 Let $G$ be a finite group and $H\leq G$. Let $g\in G$. Assume that there exists $\chi\in\Irr(S(g))$ with $\nu_m(g,\chi)\neq 0$. Then the double coset of $g$ contains an element $g'$ with $(g')^m\in S(g)$.
\end{Rem}

\begin{Prop}\label{Prop:nutwo-gen}
  Let $G$ be a finite group, and $H\leq G$.
  Then all objects of $\Cat GH$ that have nonzero Frobenius-Schur indicator are of the form $\Obj g\chi$ with $g^2\in H$ and $\ord(g)$ a power of two.

  Assume that $g\in G\setminus H$ satisfies $g^2\in H$. Let $S=\Stab_H(gH)$, and $\chi$ an irreducible character of $S$. Then
  \begin{enumerate}
    \item For $x\in H$ we have $(gx)^2\in H\Leftrightarrow (gx)^2\in S\Leftrightarrow x\in S$.
    \item $g\adhit S=g\inv\adhit S=S$.
    \item $\hat S:=S\sqcup gS=S\sqcup Sg\leq G$ and $[\hat S\colon S]=2$.
    \item $\nu_2(g,\chi)=\nu^g(\chi)$ is the twisted indicator of $\chi$ with respect to $g\in \hat S$ in the sense of \cite{MR1078503}.
  \end{enumerate}
  In particular we have
    \begin{align}
        \nu_2(g,\chi)&=\frac 1{|S|}\sum_{x\in S}\chi((gx)^2)\label{eq:18}\\
        &=\frac1{|S|}\sum_{x\in S}\chi(x(g\adhit x)g^2)\label{eq:19}\\
        &=\frac1{|S|}\sum_{x\in S}\chi((g^{- 1}\adhit x)xg^2)\label{eq:20}\\
        &=\frac1{|S|}\sum_{x\in\hat S}\chi(x^2)-\nu_2(\chi)\label{eq:21}\\
        &=\nu_2\left(\Ind_S^{\hat S}(\chi)\right)-\nu_2(\chi)\label{eq:22}\\
        &=\label{eq:23}
        \begin{cases}
          2\nu_2(\hat\chi)-\nu_2(\chi)&\text{if }g\adhit\chi=\chi\\
          \nu_2(\hat\chi)-\nu_2(\chi)&\text{if }g\adhit\chi\neq\chi.
        \end{cases}
      \end{align}
      where $\hat\chi$ is an irreducible  character of $\hat S$ whose restriction contains $\chi$.
\end{Prop}
\begin{proof}
  We already know that we can assume $g^2\in H$. Let $\ord(g)=2^k(2\ell+1)$. Then $g^{2\ell+1}=g(g^2)^\ell\in gH\subset HgH$ and $(g^{2\ell+1})^{2^k}=e$.

  For any $m$, $(gx)^m\in H\Rightarrow (gx)^m\in S$ was already observed in \cite{Sch:CHFSIFCCIFG}. Now for $x\in H$
  \begin{equation*}
    gxgx\in H\Leftrightarrow gxg\in H\Leftrightarrow xg\in g\inv H=g\inv g^2H=gH\Leftrightarrow x\in g\adhit H.
  \end{equation*}

  $g\adhit S=g\adhit(H\cap (g\adhit H))=(g\adhit H)\cap (g^2\adhit H)=(g\adhit H)\cap H=S$ holds since $g^2\in H$, and the assertions on $\hat S$ are a direct consequence of this and $g^2\in S$.

  Now \eqref{eq:18}, which is the definition of the twisted indicator in \cite{MR1078503}, is a direct consequence of \eqref{eq:24}. The two versions \eqref{eq:19},\eqref{eq:20} follow from $gxgx=(g\adhit x)g^2x=g^2(g\inv\adhit x)x$. Since $\hat S$ is the disjoint union of $S$ and $gS$, we obtain \eqref{eq:21}, which is also implicit in \cite{MR1078503}.

  \eqref{eq:22} is a special case of \cite[Lemma 2.1]{MR1078503}, we prove it for completeness: By a standard formula for induced characters, $\Ind_S^{\hat S}(\chi)(y)=\chi(y)+\chi(g\adhit y)$ for $y\in S$. Now $x^2\in S$ for all $x\in\hat S$, so that
  \begin{align*}
    \nu_2\left(\Ind_S^{\hat S}(\chi)\right)&=\frac1{|\hat S|}\sum_{x\in\hat S}\Ind_S^{\hat S}(\chi)(x^2)\\
    &=\frac1{|\hat S|}\sum_{x\in\hat S}(\chi(x^2)+\chi(g\adhit x^2))\\
    &=\frac12\left(\frac1{|S|}\sum_{x\in\hat S}\chi(x^2)+\frac1{|S|}\sum_{x\in\hat S}\chi((g\adhit x)^2))\right)\\
    &=\frac1{|S|}\sum_{x\in\hat S}\chi(x^2).
  \end{align*}
  The last formula \eqref{eq:23} is \cite[Lemma 2.3]{MR1078503}, it follows from \eqref{eq:22} by the standard analysis of the induced character.
\end{proof}

When $g$ in \cref{Prop:nutwo-gen} has order two, the twisted Frobenius-Schur indicators specialize, naturally, to the indicators twisted by a group automorphism that were singled out already in \cite{MR1078503}, and have become somewhat better known than the general ones:
\begin{Cor}\label{Cor:nutwo-spec}
  Let $H\leq G$ as above. Assume that $g\in G\setminus H$ satisfies $g^2=e$, put $S=\Stab_H(gH)$. For any character $\chi$ of $S$ we have
  \begin{equation}
    \label{eq:17}
    \nu_2(g,\chi)=\nu^\tau(\chi)
  \end{equation}
  with $\tau\colon S\to S$ given by $\tau(x)=g\adhit x$. Thus the Frobenius-Schur indicator of the object $\Obj g\chi$  is the $\tau$-twisted indicator of the group character $\chi$.

  If $\Stab_H(gH)\subset C_G(g)$, then simply $\nu_2(g,\chi)=\nu_2(\chi)$.

  If $g'=gu$ with $u\in N_G(H)$, $u^2=e$, $gu=ug$, and $\Stab_H(gH)\subset C_G(g)$, then $\Stab_H(g'H)=\Stab_H(gH)$ and $\nu_2(g',\chi)=\nu^\tau(\chi)$, where $\tau$ is conjugation by $u$.
\end{Cor}

We also note the following ``invariance'' of indicators with respect to elements commuting with $H$:
\begin{Rem}\label{rem:1}
  If $u\in G$ satisfies $H\subset C_G(u)$, then $S(u\adhit g)=S(g)$ and $\nu_m(u\adhit g,\chi)=\nu_m(g,\chi)$ for any $g\in G$, $\chi\in\Irr(S(g))$, and $m\in\Z$.

  If also $g\in C_G(u)$ and $u^m=e$, then $S(ug)=S(g)$ and $\nu_m(ug,\chi)=\nu_m(g,\chi)$.
\end{Rem}

\begin{Prop}
  \label{thm:1}
  Let $G$ be a finite group and $H\subgroup F\subgroup G$. Let $t\in G$ satisfy $t^2=e$, $F\cap t\adhit H\subset H$, and $H':=\Stab_H(tH)\subset C_G(t)$. Let $f\in F$ satisfy $f^2\in H$ and $ft=tf$.

  Then $\Stab_H(tfH)=\Stab_{H'}(fH')$, and
  \begin{equation*}
    \nu_2(tf,\chi,H)=\nu_2(f,\chi,H').
  \end{equation*}
\end{Prop}
\begin{proof}
  If $x\in H'\cap g\adhit H'$, then $x\in H\cap t\adhit H$, and $x\in f\adhit H\cap ft\adhit H$, in particular $x\in H\cap ft\adhit H$. Conversely, if $x\in H\cap tf\adhit H$, then $t\adhit x\in f\adhit H\subset F$, and thus $t\adhit x\in H$ by assumption. Thus $x\in H'$. Moreover, $f\inv\adhit x\in F$ and $tf\inv\adhit x\in H$ implies $f\adhit x\in F\cap t\adhit H\subset H$; already $tf\inv x\in\adhit H$ and so $x\in H'\cap f\adhit H'$. We have shown $S:=\Stab_H(tfH)=\Stab_{H'}(fH')$. For the indicators it only remains to observe that
  \begin{equation*}
    \nu_2(tf,\chi,H)=\frac 1{|S|}\sum_{x\in S}\chi((tfx)^2)
    =\frac 1{|S|}\sum_{x\in S}\chi((fx)^2)=\nu_2(f,\chi,H').
  \end{equation*}
\end{proof}

\section{Examples}
\label{sec:examples}

A particular case of the following example (with $n=p=7$, $\ell=5$, and $g=(56)$) occurs in figure 3 (and figure 2) in \cite{Sch:CHFSIFCCIFG}; the conspicuous block of zeroes was pointed out to me by Joe Timmer.
\begin{Expl}\label{ex:1}
  Let $\S_\ell\subset \S_n$ be the subgroup of the permutation group on $n$ letters that keeps $n-\ell$ letters (say, the last $n-\ell$) fixed. Let $p$ be a prime such that $\ell+p>n$. Let $g\in \S_n\setminus \S_\ell$ such that $|\supp(g)\cap \{\ell+1,\dots,n\}|<p-\ell$. Then $\nu_p(g,\chi)=0$ for all $\chi\in\Irr(S(g))$.
\end{Expl}
\begin{proof}
  Otherwise by \cref{rem:2} the double coset of $g$ would have to contain an element $g'$ whose $p$-th power is the identity.  By the conditions on $p$, $g'$ is a $p$-cycle. But since $g$ moves less than $p-\ell$ points greater than $\ell$, the same is true for $g'$ in the $\S_\ell$-double coset of $g$. Therefore, $g'$ moves less than $p$ points, a contradiction.
\end{proof}

The following result generalizes \cite[Thm.~5.1]{2014arXiv1412.4725}:
\begin{Thm}\label{Thm:SellinSn}
  Let $\S_\ell\subset \S_n$ in the natural way. Consider $g\in \S_n$. Then either $\nu_2(g,\chi)=0$ for all $\chi$, or $\nu_2(g,\chi)=1$ for all $\chi$.
\end{Thm}
\begin{proof}
  We can assume $g^2\in \S_\ell$. In the cycle decomposition of $g$, if $i\in\{1,\dots,n\}$ occurs in a cycle of length not two, then $i$ is not fixed by $g^2$. Thus letters beyond $\ell$ only occur in transpositions in the cycle decomposition of $g$. Without changing the double coset of $g$, we can strip it of any cycles in $\S_\ell$. So finally $g$ is a product of transpositions, each of which contains at least one element greater than $\ell$. As noted in \cite{Sch:CHFSIFCCIFG}, the stabilizer of $g$ is then the symmetric group on those letters in $\{1,\dots,\ell\}$ that are not moved by $g$. In particular $\Stab_{\S_\ell}(g\S_{\ell})\subset C_{\S_n}(g)$ and thus $\nu_2(g,\chi)=\nu_2(\chi)=1$.
\end{proof}
\begin{Rem}\label{rem:3}
  The statement of \cref{Thm:SellinSn} is that among the double cosets of $\S_\ell$ in $\S_n$ there are some where all indicators of the associated simples in $\C(\S_n,\S_\ell)$ are zero (following Timmer \cite{2014arXiv1412.4725} we might call them \emph{null indicator double cosets}, or their elements \emph{null indicator elements}), and others where all the indicators equal one. It takes only slightly more work to decide, for an arbitrary element $\sigma\in \S_n$, whether it belongs to a null indicator double coset or not. First, using the observations
  \begin{align*}
    (i,a_1,\dots,a_r,j,b_1,\dots,b_s)(ij)&=(i,b_1,\dots,b_s)(j,a_1,\dots,a_r)\\
    (i,j,a_1,\dots,a_r)(ij)&=(i,a_1,\dots,a_r)
  \end{align*}
  for $i,j\in\{1,\dots,\ell\}$ we can multiply $g$ repeatedly from the right by transpositions in $\S_\ell$ to obtain a representative with the property that each of the cycles in its disjoint cycle decomposition contains at most one element from $\{1,\dots,\ell\}$. Let now $\sigma$ and $\sigma'$ be two permutatins with this property, and $s,t\in \S_\ell$ such that $s\sigma t\inv=\sigma'$. Let $i\leq\ell$ such that $\sigma(i)>\ell$, and let $\lambda$ be the length of the $\sigma$-orbit of $i$. Then $j=t(i)$ has $\sigma'(j)=s\sigma t\inv(j)>\ell$. On elements beyond $\ell$ the two permutations $\sigma$ and $\sigma'$ act in the same way, so that the $\sigma'$-orbit of $j$ contains $j,\sigma(i),\sigma^2(i),\dots,\sigma^{\lambda-1}(i)$. Now $\sigma'\sigma^{\lambda-1}(i)=s\sigma^\lambda(i)=s(i)$ also belongs to the $\sigma'$-orbit of $j$, and $s(i)\leq\ell$, and so $s(i)=j$ by the assumption on $\sigma'$. On the other hand, if $i\leq\ell$ is not moved by $\sigma$, then $j=t(i)$ is mapped to $s(i)$ by $\sigma'$, and since $\sigma'$ cannot move $j\leq\ell$ to a different element in $\{1,\dots,\ell\}$, we see $s(i)=j$. All in all $s=t$.

  We have shown that every $\S_\ell$-double coset in $\S_n$ contains an element such that every cycle of its cycle decomposition contains at most one element from $\{1,\dots,\ell\}$, and that this representative is unique up to conjugation by an element in $\S_\ell$ (that is, renumbering of the elements in $\{1,\dots,\ell\}$ that occur.

  Now it so happens that the elements found in the proof of \cref{Thm:SellinSn} to represent the non-null indicator double cosets are among the representatives we have now found for all the double cosets. Thus $\sigma$ is in a null indicator double coset if and only if its (easily computable) representative is not a product of disjoint transpositions.

  It may also be worth noting that, provided $2\ell\geq n$, the number of double cosets of $\S_\ell$ in $\S_n$, as well as the number of null indicator double cosets among them, only depends on the difference $n-\ell$.

  For $\S_{n-2}\subset \S_n$ we have already listed the seven double cosets in \cite{Sch:CHFSIFCCIFG}; two of them are null indicator double cosets.

  For $\S_{n-3}\subset \S_n$ it suffices to consider $\S_3\subset \S_6$. We list the coset representatives (with $\bullet$ representing elements in $\{1,\dots,3\}$):
  \begin{itemize}
  \item The permutations from $\S_6$ fixing $1,2,3$ represent six double cosets, among which the identity and the three transpositions represent the non-null indicator double cosets.
  \item $(\bullet,i)$ and $(\bullet,i)(j,k)$ with $i\in\{4,5,6\}$ and $\{i,j,k\}=\{4,5,6\}$ represent six non-null indicator double cosets.
  \item $(\bullet,i)(\bullet,j)$ with $i,j\in\{4,5,6\}$ and $i\neq j$ represent three non-null indicator double cosets.
  \item $(\bullet,4)(\bullet,5)(\bullet,6)$ represent(s) one non-null indicator double coset.
  \item $(\bullet,i,j)$ with $i,j\in\{4,5,6\}$, $i\neq j$, represent six null indicator double cosets.
  \item $(\bullet,i)(\bullet,j,k)$ with $\{i,j,k\}=\{4,5,6\}$ represent six null indicator double cosets.
  \item $(\bullet,i,j,k)$ with $\{i,j,k\}=\{4,5,6\}$ represent six null indicator double cosets.
  \end{itemize}
  Thus, in total, we get 34 double cosets, among which are 20 null indicator double cosets.

  For $\S_{n-4}\subset \S_n$, say $n=8$, we will be even more telegrammatic, listing almost only cycle shapes with $\bullet$ indicating elements in $\{1,\dots,4\}$ and $\times$ elements in $\{5,\dots,8\}$.
  We begin by listing the representatives of non-null indicator double cosets:
  \begin{itemize}
  \item The ten elements of order one or two fixing $1,\dots,4$.
  \item $(\bullet,\times)\sigma$, where $\sigma$ is among the four elements of square one that fix $1,\dots,4$ and one more element; this gives a total of $16$ double cosets.
  \item $(\bullet,\times)(\bullet,\times)$ represent six double cosets.
  \item $(\bullet,\times)(\bullet,\times)(\times,\times)$ represents six.
  \item $(\bullet,\times)(\bullet,\times)(\bullet,\times)$ represents four.
  \item $(\bullet,\times)(\bullet,\times)(\bullet,\times)(\bullet,\times)$ represents one.
  \end{itemize}
  This exhausts the $43$ non-null indicator double cosets. For the null indicator ones we list
  \begin{itemize}
  \item The $14$ elements fixing $1,\dots,4$ whose square is not one,
  \item $(\bullet,\times)(\times,\times,\times)$ for $8$ double cosets,
  \item $(\bullet,\times,\times)$ for $12$ double cosets,
  \item $(\bullet,\times,\times)(\times,\times)$ for $12$,
  \item $(\bullet,\times,\times)(\bullet,\times)$ for $24$,
  \item $(\bullet,\times,\times)(\bullet,\times,\times)$ for $12$,
  \item $(\bullet,\times,\times,\times)$ for $24$,
  \item $(\bullet,\times,\times,\times)(\bullet,\times)$ for $24$, and
  \item $(\bullet,\times,\times,\times,\times)$ for $24$.
  \end{itemize}
  These sum up to $154$ null indicator and a total of $197$ double cosets.

  It is perhaps not entirely unreasonable to conjecture that with increasing $k$ (and $n$ large enough, that is $n\geq 2k$), the proportion of null indicator double cosets among all double cosets might tend to $1$.
\end{Rem}

The following observation is likely known:
\begin{Lem}\label{thm:2}
  Let $\sigma\in\S_n\setminus\A_n$. The twisted indicator $\nu_2^\sigma(\chi)$ of an irreducible character $\chi$ of $\A_n$ is $0$ or $1$.
\end{Lem}
\begin{proof}
  $\nu_2^\sigma(\chi)=\nu_2\left(\chi\upharpoonright{\S_n}\right)-\nu_2(\chi)$, and $\S_n$ is totally orthogonal.
\end{proof}
The following was conjectured, with a near-complete proof, in \cite[Conj.5.6]{2014arXiv1412.4725}:
\begin{Thm}\label{thm:4}
  All simple objects of $\Cat{\S_n}{\A_n}$ have Frobenius-Schur indicator $0$ or $1$.
\end{Thm}
\begin{proof}
  Two suitable double coset representatives are $e$ and $(12)$. Since $\nu_2(e,\chi)=\nu_2(\chi)$, only the case $g=(12)$ is interesting. Clearly $S:=\Stab_{\A_n}((12)\A_n)=\A_n$. Thus $\nu_2((12),\chi)$ is zero or one by \cref{thm:2}.
\end{proof}

More generally, the same result holds for smaller alternating groups embedded in the canonical way in $\S_n$.
\begin{Thm}\label{thm:3}
  All simple objects of $\Cat{\S_n}{\A_\ell}$ for $\ell<n$ have Frobenius-Schur indicator $0$ or $1$.
\end{Thm}
\begin{proof}
  It suffices to consider $\Obj g\chi$ with $g\in \S_n\setminus \A_\ell$ and $g^2\in \A_\ell$. As in the proof of \cref{Thm:SellinSn}, we see that elements beyond $\ell$ can only occur in transpositions in the cycle decomposition of $g$, and multiplying $g$ by an element of $\A_\ell$ we can achieve that $g$ is a product of transpositions, at most one of which is in $\S_\ell$. For $x\in \A_\ell$ to be in $g\adhit \A_\ell$ means $g\inv\adhit x$ does not move letters beyond $\ell$, thus that $x$ does not move one of the letters paired with an element beyond $\ell$ in a transposition of $x$. Write $\S_k'$ for the subgroup of $\S_n$ that keeps $1,\dots,k$ fixed. Without loss of generality we can assume $S=\Stab_{\A_\ell}(g\A_\ell)=\A_k$ for $k\leq\ell$, and either $g\in \S_k'$ or $k\geq 2$ and $g$ is the product of $(12)$ and a product of transpositions in $\S_k'$.

  If $g\in \S_k'$, then $\Stab_{\A_\ell}(g\A_\ell)\subset C_{\S_n}(g)$ and $\nu_2(g,\chi)=\nu_2(\chi)\in\{0,1\}$.

  If $g=(12)g'$ with $g'\in \S_k'$, then $\nu_2(g,\chi)=\nu_2^\tau(\chi)\in\{0,1\}$, with the automorphism $\tau$ of $\A_\ell$ given by conjugation with $(12)$.
\end{proof}

The following was proved in \cite{MR2471448} in the case where $n$ is prime; we make no attempt at reproducing the exact count of simples with the respective indicators, nor the asymptotics of the proportions of the two indicator values among the simples for large primes $n$ analyzed there.
\begin{Thm}\label{thm:5}
  If $n$ is not divisible by $4$, then every simple object of $\Cat{\S_n}{C_n}$ has Frobenius-Schur indicator $0$ or $1$, where we consider $C_n$, the cyclic group of order $n$ generated by an $n$-cycle in $\S_n$.
\end{Thm}
\begin{proof}
  Consider $g\in \S_n\setminus C_n$ such that $g^2\in C_n$, and the order of $g$ is a power of two. Since $n$ is not divisible by four, the order of $g^2$
    can only be one or two. In the latter case the order of $g$ is four, and since $g^2\in C_n$, $g^2$ has no fixed points, and $g$ is a product of four-cycles without a fixed point. This contradicts the assumption that four does not divide $n$. Thus we have $g^2=()$. The stabilizer subgroup $S=\Stab_{C_n}(gC_n)$ is cyclic of order $m$ dividing $n$. Denote $t$ a generator of $S$. Then $g\adhit t=t^\ell$ for some $\ell$. Let the irreducible character $\chi$ of $S$ be given by $\chi(t)=\zeta$, where $\zeta$ is an $m$-th root of unity. Then
  \begin{equation*}
    \nu_2(g,\chi)=\frac1{|S|}\sum_{x\in S}\chi(x(g\adhit x))=\frac 1m\sum_{k=0}^{m-1}\chi(t^kt^{\ell k})=\frac 1m\sum_{k=0}^{m-1}(\zeta^{1+\ell})^k
  \end{equation*}
  is zero or one, since $\zeta^{1+\ell}$ is an $m$-th root of unity.
\end{proof}
The following example shows that if we omit the condition on $n$, the conclusion may fail. Note though that (according to GAP calculations) no indicator of a simple in $\Cat{\S_8}{C_8}$ is negative. We don't know a necessary and sufficient condition for the indicator value $-1 $ to occur.
\begin{Expl}\label{ex:2}
  Consider $t=(1\;2\;3\;4\;5\;6\;7\;8\;9\;10\;11\;12)$ and $H=C_{12}=\langle t\rangle\subset \S_{12}=G$. Put $g=(1\;2\;7\;8)(3\;11\;9\;5)(4\;12\;10\;6)$. We see that
  \begin{align*}
    g\inv\adhit t&=(1\;5\;6\;9\;10\;2\;7\;11\;12\;3\;4\;8)\not\in H\\
    g\inv\adhit t^2&=(1\;6\;10\;7\;12\;4)(2\;11\;3\;8\;5\;9)\not\in H\\
    g\inv\adhit t^3&=(1\;9\;7\;3)(2\;12\;8\;6)(3\;1\;9\;7)\not\in H\\
    g\adhit t^4&=(1\;10\;12)(2\;3\;5)(4\;6\;7)(8\;9\;11)\not\in H, \text{ but}\\
    t^6=g^2&=(1\;7)(2\;8)(3\;9)(4\;10)(5\;11)(6\;12),
  \end{align*}
  so that $S=H\cap g\adhit H$ is cyclic of order two, generated by $g^2$. For a character $\chi$ of $S$ we have
  \begin{equation*}
    \nu_2(g,\chi)=\frac12\sum_{x\in S}\chi(g^2x^2)=\chi(g^2)
  \end{equation*}
  which equals $-1$ if $\chi$ is the nontrivial irreducible character of $S$.
\end{Expl}

The following results settle \cite[Conj.5.3]{2014arXiv1412.4725}. More precisely, we prove it up to a finite set of exceptions. To understand this set, we note that the alternating group $\A_n$ is ambivalent if and only if $n\in\{1,2,5,6,10,14\}$, see \cite{MR0240187}. By \cite[Cor.~3.8]{MR963550}, this implies that those particular alternating groups are totally orthogonal, while we have already used that no alternating group admits a representation with indicator $-1$ (by \cite[Prop.~3.9]{MR963550}). 
\begin{Thm}\label{thm:6}
  For $n\geq 4$, the alternating group $\A_n$ contains an isomorphic copy of $\S_{n-2}$, namely $\tilde \S_{n-2}=\A_n\cap((12)\S_2'\sqcup \S_2')$.

  If $n\in \{4,5,6,9,10,14,18\}$ the indicators of all the simple objects of the category $\Cat{\S_n}{\tilde \S_{n-2}}$ are $0$ or $1$. For all other $n$ the subcategory $\Cat{\A_n}{\tilde \S_{n-2}}$ contains a simple object with Frobenius-Schur indicator $-1$.
\end{Thm}
\begin{proof}First, we consider $H=\tilde \S_{n-2}$ and $G=\S_n$.
    Let $g\in \S_n\setminus\tilde \S_{n-2}$ satisfy $g^2\in \tilde \S_{n-2}$. This means that $g^2$ either fixes or switches $1$ and $2$.

  If $g^2$ fixes $1$ and $2$, then both can only occur in transpositions of the cycle decomposition of $g$. If they are in the same cycle then $g\in H$.

  If only one of $1$ and $2$ occurs in a transposition, say $2$, then after multiplying with elements of $\A_2':=\S_2'\cap \A_n\subset H$ we can assume that $g=(23)g'$ where $g'\in\{(),(45)\}$.

  If both $1$ and $2$ occur in transpositions, then after multiplying with elements of $\A_2'$ we can assume that $g=(13)(24)g'$ with $g'\in\{(),(56)\}$.

  Now if $g^2$ switches $1$ and $2$, then $1$ and $2$ occur in a four-cycle in $g$, and $n\geq 6$ since $g$ is even. After multiplying with elements of $\A_2'$ we can assume that this four-cycle is $(1423)$. But $(1423)(12)(56)=(13)(24)(56)$, so that after multiplying further with elements of $\A_2'$ we can assume that $g=(13)(24)g'$ with $g'\in\{(),(56)\}$.

  We have thus four cases to treat:
  \begin{itemize}
  \item $g=(23)$, with $\Stab_H(gH)=\A_3'\subset C_G(g)$, and thus $\nu_2(g,\chi)=\nu_2(\chi)\in\{0,1\}.$
  \item $g=(23)(45)$ with $S=\Stab_G(gH)=\A_3'$ and $\hat S=(\S_3'\cup(23)\S_3')\cap \A_n\cong \S_{n-3}$. Thus $\nu_2(g,\chi)=\nu_2\left(\Ind_{S}^{\hat S}(\chi)\right)-\nu_2(\chi)\in\{0,1\}$.
  \item $g=(13)(24)$. Let $x\in H$. Then $g\adhit x\in H$ iff $g\adhit x$ fixes or switches $1$ and $2$, iff $x$ fixes or switches $3$ and $4$. Thus
  \begin{equation*}
    S=\Stab_G(gH)=\A_n\cap\left(S'_4\sqcup (12)S'_4\sqcup (34)S'_4\sqcup (12)(34)S'_4\right).
  \end{equation*}
  Note that conjugation by $g$ switches the two middle terms in the disjoint union, without affecting the factors in $\S_4'$. Writing $T=\A_n\cap((34)\S_4'\sqcup \S_4')\cong \S_{n-4}$ we have a group isomorphism $T\times\langle(12)(34)\rangle\cong S$ given by multiplication,
    and thus an isomorphism $\S_{n-4}\times \{\pm 1\}\cong S$ under which conjugation by $g$ corresponds to the automorphism
    \begin{align*}
      \tau\colon \S_{n-4}\times \{\pm 1\}&\to \S_{n-4}\times\{\pm 1\}\\
      (\sigma,x)&\mapsto (\sigma,x\epsilon(\sigma))
    \end{align*}
    where $\epsilon(\sigma)$ denotes the sign. An irreducible character $\chi$ of $S$ can be identified with the product $\theta\times\eta$ of a character $\theta$ of $\S_{n-4}$ and a character $\eta$ of $\{\pm 1\}$.
    \begin{align*}
      \nu_2(g,\chi)&=\nu^\tau(\theta\times\eta)\\
      &=\frac1{2(n-4)!}\sum_{\substack{\sigma\in \S_{n-4}\\x=\pm 1}}(\theta\times\eta)\left((\sigma,x)\tau(\sigma,x)\right)\\
      &=\frac1{2(n-4)!}\sum_{\substack{\sigma\in \S_{n-4}\\x=\pm 1}}\theta(\sigma^2)\eta(x^2\epsilon(\sigma))\\
      &=\frac1{(n-4)!}\sum_{\sigma\in \S_{n-4}}\eta\epsilon(\sigma)\theta(\sigma^2)
    \end{align*}
    So that $\nu_2(g,\chi)=\nu_2(\chi)$ if $\eta$ is the trivial character.  If $\eta$ is the nontrivial character, then further
    \begin{align*}
      \nu_2(g,\chi)&=\frac1{(n-4)!}\sum_{\sigma\in \S_{n-4}}\epsilon(\sigma)\theta(\sigma^2)\\
      &=\frac1{(n-4)!}\sum_{\sigma\in \A_{n-4}}\theta(\sigma^2)-\frac 1{(n-4)!}\sum_{\sigma\in \S_{n-4}\setminus \A_{n-4}}\theta(\sigma^2)\\
      &=\frac1{(n-4)!}\sum_{\sigma\in \A_{n-4}}\theta(\sigma^2)-\frac 1{(n-4)!}\sum_{\sigma\in \A_{n-4}}\theta((\sigma(12))^2)\\
      &=\frac 12\nu_2(\theta|_{\A_{n-4}})-\frac 12\nu^{(12)}(\theta|_{\A_{n-4}})\\
      &=\frac12\nu_2(\theta|_{\A_{n-4}})-\frac12\left(\nu_2(\theta|_{\A_{n-4}}\upharpoonright^{\S_{n-4}})-\nu_2(\theta|_{\A_{n-4}})\right)\\
      &=\nu_2(\theta|_{\A_{n-4}})-\frac12\nu_2(\theta|_{\A_{n-4}}\upharpoonright^{\S_{n-4}})\\
      &=\nu_2(\theta|_{\A_{n-4}})-1.
    \end{align*}
    Now if $\A_{n-4}$ is ambivalent, then $\nu_2(\theta|_{\A_{n-4}})$ equals one or two, but if not, there exists $\theta$ where it vanishes, giving $\nu_2(g,\chi)=-1$.

    Note that since $g\in \A_n$, the object with indicator $-1$ that we have (possibly) found is in the subcategory $\C(\A_n,\tilde \S_{n-2})$.
  \item $g=(13)(24)(56)$. By the same reasoning as before,
    $S(g)=\A_n\cap\left(\S_4'\disj (12)S'_4\disj(34)S'_4\disj(12)(34)S'_4\right),$
    but this time conjugation by $g$ switches the two middle terms while conjugating the $\S_4'$-factors involved by $(56)$. We thus get an isomorphism $\S_{n-4}\times\{\pm 1\}\cong S(g)$ under which conjugation by $g$ corresponds to
        \begin{align*}
      \tau\colon \S_{n-4}\times \{\pm 1\}&\to \S_{n-4}\times\{\pm 1\}\\
      (\sigma,x)&\mapsto ((12)\sigma(12),x\epsilon(\sigma)).
    \end{align*}
    With notations as above
        \begin{align*}
      \nu_2(g,\chi)&=\nu^\tau(\theta\times\eta)\\
      &=\frac1{2(n-4)!}\sum_{\substack{\sigma\in \S_{n-4}\\x=\pm 1}}(\theta\times\eta)\left((\sigma,x)\tau(\sigma,x)\right)\\
      &=\frac1{2(n-4)!}\sum_{\substack{\sigma\in \S_{n-4}\\x=\pm 1}}\theta(\sigma(12)\sigma(12))\eta(x^2\epsilon(\sigma))\\
      &=\frac1{(n-4)!}\sum_{\sigma\in \S_{n-4}}\eta\epsilon(\sigma)\theta(\sigma^2)
    \end{align*}
    which is the same expression as for the previously treated case. Note that since now $g$ is odd, the object with indicator $-1$ we have (possibly) found now belongs to $\C(\S_n,\tilde\S_{n-2})\setminus\C(\A_n,\tilde\S_{n-2})$.
  \end{itemize}
\end{proof}
\begin{Thm}\label{thm:7}
  Let $n\geq 4$. If $n\in \{4,5,6,10\}$, then the indicators of all the simples in the category $\C(\S_{n+1},\tilde\S_{n-2})$ are zero or one. For all other $n$ the subcategory $\C(\A_{n+1},\tilde\S_{n-2})$ contains a simple object with Frobenius-Schur indicator $-1$.
\end{Thm}
\begin{proof}
 Consider $ g\in \S_{n+1}\setminus \S_n $ such that
$ g^2\in\tilde\S_{n-2} $.
This means that $n+1$ occurs in a transposition of $g$; after
conjugating with an element of $\tilde S_{n-2}$ we can assume this
is one of the transpositions $(n,n+1)$ and $(2,n+1)$.

  Assume first that $g=\sigma(2,n+1)$ with $\sigma\in\S_n$ fixing $2$. Let $x\in \tilde\S_{n-2}$. If $g\inv\adhit x\in \tilde\S_{n-2}$ then $g\inv\adhit x$ fixes $n+1$, and thus $x$ fixes $2$, and thus $x$ fixes $1$, and $x\in\A_2'$.

  If $\sigma$ fixes $1$, then $\sigma\in\S_2'$, any $x\in\A_2'$ belongs to $\Stab_{\tilde\S_{n-2}}(g\tilde\S_{n-2})$, and any $\nu_2(g,\chi)$ is either an indicator of $A_2'$ or a twisted indicator $\nu^\sigma(\chi)$, hence at any rate zero or one.

  If $\sigma$ does not fix $1$, then we can assume $g=(13)(2,n+1)\sigma'$ with $\sigma'\in\S_3'$. If $x\in\tilde\S_{n-2}$ satisfies $g\inv\adhit x\in\S_{n-2}$, then $x$ fixes $1,2$ and $3$, so that $x\in\A_3'$. Thus any $\nu_2(g,\chi)$ is either an indicator of $\A_3'$ or a twisted indicator $\nu^{\sigma'}(\chi)$, hence equals zero or one.

  Summing up, the case that $n+1$ occurs in the transposition $(2,n+1)$ in $g$ does not yield new objects with indicator $-1$.
  
  Assume now that $g=f(n,n+1)$, with $f\in \S_{n-1}$. Then we are in the situation of \cref{thm:1}, with $F=S_n$, $t=(n,n+1)$ and $C_{S_{n+1}}(t)=S_{n-1}$; therefore, the indicator values for such elements are exactly the indicator values for the simples in $\C(S_{n-1},\tilde S_{n-3})$ associated to $f$ running through the elements of $S_{n-1}$. If $n=4$ then $\tilde S_{n-3}$ is trivial, while for $n>4$ the preceding theorem shows that the indicator value $-1$ is possible if and only if $n-1\not\in\{4,5,6,9,10,14,18\}$, or $n\not\in\{5,6,10,11,15,19\}$, and the element $x$ can then be chosen odd, so that $g$ is even. $n=5,6,10$ are the only cases where there is not already an object with Frobenius-Schur indicator $-1$ in $\C(\A_n,\tilde\S_{n-2})\subset\C(\A_{n+1},\tilde\S_{n-2})$.
\end{proof}
\begin{Thm}\label{thm:8}
  Let $n\geq 4$ and $k\geq 2$. For $n\in\{4,5,6\}$ the indicators of all the simples in the category $\C(\S_{n+k},\tilde\S_{n-2})$ are zero or one. For all other $n$ the subcategory $\C(\A_{n+k},\tilde\S_{n-2})$ contains a simple object with Frobenius-Schur indicator $-1$.
\end{Thm}
\begin{proof}
  Consider $H=\tilde\S_8\subset F=\S_{10}\subset G=\S_{12}$.

  With $t=(9,11)(10,12)\in\S_{12}$, $f=(13)(24)$, and $g=ft$ we are in the situation of \cref{thm:1}.
  Thus the indicators of the objects in $\C(\S_{12},\tilde\S_8)$ associated to $g\in S_{12}$ are the same as the indicators of the objects in $\C(\S_8,\tilde\S_6)$ associated to $f$; we have seen that $-1$ occurs among them.

  It remains to check that the indicator value $-1$ is impossible for $n\in\{4,5,6\}$; we will do this by induction on $k$. Assuming $\nu_2(g,\chi)\neq -1$ for $g\in\S_{n+k-1}$, we can consider $g\in\S_{n+k}\setminus\S_{n+k-1}$ such that $g^2\in H:=\tilde S_{n-2}$ and $\chi\in\Irr(S(g))$.
  In particular $g$ switches $n+k$ with an element $i\in\{1,\dots,n+k-1\}$.

  If $i\in\{n+1,\dots,n+k-1\}$ then $\nu_2(g,\chi)$ is the indicator of a simple in $\C(\S_{n+k-1},\tilde\S_{n-2})$ \cref{rem:1}, hence equals zero or one by the induction hypothesis.

  If $i\in\{3,\dots,n\}$, we can assume $i=n$. Thus $g=ft$ with $t=(n,n+k)$ and $f\in F=\{\sigma\in\S_{n+k-1}|\sigma(n)=n\}$. As $\Stab_H(tH)=\tilde\S_{n-3}\subset C_G(t)$ we are in the situation of \cref{thm:1}, and (up to renumbering), $\nu_2(g,\chi)$ is the indicator of a simple in $\C(S_{n+k-2},\tilde\S_{n-3})$, hence zero or one by induction hypothesis (if $n>4$, otherwise $\S_{n-3}$ is trivial).

  If $i\in\{1,2\}$, we can assume $i=2$. We then have to treat separately the subcases where $g$ fixes or moves $1$.

    If $g$ fixes $1$, then $g=tf$ with $t=(2,n+k)$ and $f\in F=\S_{n+k-1}$ fixing $1$ and $2$. We have $\Stab_H(tH)=\{\sigma\in H|\sigma(2)=2\}=A_2'\subset C_G(t)$, and by \cref{thm:1} we see that $\nu_2(g,\chi)$ is the indicator of a simple in $\C(\S_{n+k-1},\A'_2)$. These are zero or one by \cref{thm:3}.

    Finally, if $g$ moves $1$, then we can assume $g=ft$ with $t=(2,n+k)(13)$ and $f$ fixing $1,2,3$ and $n+k$. Any $x\in\Stab_H(tH)$ needs to fix $2$, hence $1$, so $t\adhit x$ fixes $3$. Conversely any $x\in H$ fixing $1,2$ and $3$ commutes with $t$, so $\Stab_H(tH)=A_3'\subset C_G(t)$. By \cref{thm:1} we conclude that $\nu_2(g,\chi)$ is the indicator of a simple in $\C(S_{n+k-1},A'_3)$, hence equals zero or one by \cref{thm:3}.
\end{proof}
\printbibliography

\end{document}